\documentclass[12pt]{amsart}
\theoremstyle{plain}
\newtheorem{thm}{Theorem}[section]
\newtheorem{theorem}[thm]{Theorem}

\newtheorem{lemma}[thm]{Lemma}
\newtheorem{corollary}[thm]{Corollary}
\newtheorem{proposition}[thm]{Proposition}
\theoremstyle{definition}
\newtheorem{remark}[thm]{Remark}

\newtheorem{notation}[thm]{Notation}
\newtheorem{definition}[thm]{Definition}

\newtheorem{example}[thm]{Example}

\newtheorem{problem}[thm]{Problem}

\numberwithin{equation}{section}

\newcommand{\sC}{{\mathcal C}}
\newcommand{\sD}{{\mathcal D}}
\newcommand{\sE}{{\mathcal E}}

\newcommand{\sK}{{\mathcal K}}

\newcommand{\sO}{{\mathcal O}}

\newcommand{\sT}{{\mathcal T}}

\newcommand{\C}{{\mathbb C}}

\newcommand{\BP}{{\mathbb P}}

\title[VMRT on double covers of projective space]{Varieties of minimal rational tangents on double covers of projective space}

\author{Jun-Muk Hwang, Hosung Kim}

\address{Korea Institute for Advanced Study, Hoegiro 87, Seoul, 130-722, Korea} \email{jmhwang@kias.re.kr, hosung@sogang.ac.kr}

\thanks{Both authors are supported
by National Researcher Program 2010-0020413 of NRF and MEST}

\begin{document}

\begin{abstract}
Let $\phi: X \to \BP^n$ be a double cover
 branched along a smooth hypersurface of degree $2m, 2 \leq m \leq n-1$.
   We study the varieties of minimal rational tangents $\sC_x
\subset \BP T_x(X)$ at a general point $x$ of $X$. We describe the
homogeneous ideal of $\sC_x$ and show that the projective
isomorphism type of $\sC_x$ varies in a maximal way as $x$ varies
over general points of $X$. Our description of the ideal of
$\sC_x$ implies a certain rigidity property of the covering
morphism $\phi$. As an application of this rigidity, we show that
any finite morphism between such double covers with $m=n-1$ must
be an isomorphism. We also prove that Liouville-type extension
property holds with respect to minimal rational curves on $X$.
\end{abstract}

\maketitle


\section{Introduction} Throughout the paper, we will work over the
field of complex numbers.  Let $X$ be a Fano manifold of Picard
number 1. For a general point $x \in X$, a rational curve through
$x$ is called a minimal rational curve if its degree with respect
to $K^{-1}_X$ is minimal among all rational curves through $x$.
Denote by $\sK_x$ the space of minimal rational curves through
$x$. The projective subvariety $\sC_x \subset \BP T_x(X)$ defined
as the union of tangent directions to members of $\sK_x$ is called
the variety of minimal rational tangents (VMRT) at $x$. The
projective geometry of $\sC_x$ plays a key role
 in understanding the geometry of $X$, often leading to a certain
 rigidity phenomenon (cf. the survey \cite{Hw01}) on $X$. This motivates
 the study of the geometry  of $\sC_x \subset \BP T_x(X)$ for various examples of $X$.

 In the current article, we study the case when $X$ is a double
 cover $\phi: X \to \BP^n, n \geq 3,$ of projective space $\BP^n$
 branched along a smooth hypersurface
 $Y \subset \BP^n$ of degree $2m, 2 \leq m \leq n-1$. Although this is
 one of the basic examples of Fano manifolds, its VMRT $\sC_x \subset \BP T_x(X)$ has not
 been described explicitly.
  Our first result
   is the following description of the defining equations of the VMRT.

 \begin{theorem}\label{t.VMRT} For a double cover $X \to \BP^n, n \geq 3,$
 branched along a smooth hypersurface of degree $2m, 2 \leq m \leq n-1$,
 the VMRT $\sC_x \subset \BP
 T_x(X)$ at a general point $x \in X$ is a smooth complete
 intersection of multi-degree $(m+1, m+2, \ldots, 2m)$.
 \end{theorem}

It is enlightening to compare Theorem \ref{t.VMRT} with the case
when $X$ is a smooth hypersurface of degree $m, 2 \leq m \leq n$,
in $\BP^{n+1}$. In the latter case, it is classical that the VMRT
at a general point is a smooth complete intersection of
multi-degree $(2,3, \ldots, m)$ (e.g. Example 1.4.2 in \cite{Hw01}
or \cite{LR}).

In the course of proving Theorem \ref{t.VMRT}, we will also prove
the following partial converse to it.

\begin{theorem}\label{t.converse} Let $Z \subset \BP^{n-1}, n \geq 3,$ be a
general complete intersection of multi-degree $(m+1, m+2, \ldots,
2m)$ with $2 \leq m \leq n-1$. Then there exists a smooth
hypersurface $Y \subset \BP^n$ of degree $2m$ such that a double
cover $X$ of $\BP^n$ branched along $Y$ has a point $x \in X$ with
its VMRT $\sC_x \subset \BP T_x(X)$ isomorphic to $Z \subset
\BP^{n-1}$. \end{theorem}

Theorem \ref{t.VMRT} and Theorem \ref{t.converse} are proved by
explicit computation for a certain choice of $Y$, based on the
fact that minimal rational curves of $X$ correspond to lines of
$\BP^n$ which have even contact order with $Y \subset \BP^n$ as
recalled in Proposition \ref{p.mrc}. Then Theorem \ref{t.VMRT} for
arbitrary smooth $Y$ can be obtained by a flatness argument.

Our explicit computation  enables us to study also the variation
of the VMRT $\sC_x$ as $x$ varies over $X$. Describing the
variation of VMRT is not an easy problem even for very simple Fano
manifolds, such as hypersurfaces in $\BP^{n+1}.$ In \cite{LR},
Landsberg and Robles proved that when $X$ is a general
hypersurface of degree $\leq n$ in $\BP^{n+1}$, the VMRT at
general points of $X$ have maximal variation. We will prove the
following analogue of their result in our setting.

\begin{theorem}\label{t.LR}
Let $Y\subset \BP^n, n \geq 4,$ be a general hypersurface  of
degree $2m, 2 \leq m \leq n-1$, and let $X$ be a double cover of
$\BP^n$ branched along $Y$. Then the family of VMRT's
$$\{ \sC_x \subset \BP T_x(X) \ |  \mbox{ general } x \in X \}$$ has
maximal variation. More precisely, for a general point $x \in X$,
choose a trivialization of $\BP T(U) \cong \BP^{n-1} \times U$ in
a neighborhood $U$ of $x$.  Define a morphism $\zeta: U \to {\rm
Hilb}(\BP^{n-1})$ by $\zeta(y) := [\sC_y]$ for $y \in U$. Then the
rank of $d\zeta_x$ is $n$ and the intersection of the image of
$\zeta$ and the $GL(n,\C)$-orbit of $\zeta(x)$ is isolated at
$\zeta(x)$.
\end{theorem}

The condition $n \geq 4$ in Theorem \ref{t.LR} excludes the case
of $(n,m) = (3,2)$. It is likely that the statement of Theorem
\ref{t.LR} holds also for this case. However, this case  seems to
require much more complicated computation.

 As mentioned at the beginning, the geometry of $\sC_x$ often
leads to a certain rigidity result. What  rigidity phenomenon does
our description of $\sC_x$ exhibit?  The double covering morphism
$\phi: X \to \BP^n$ sends members of $\sK_x$ to lines. We show
that this property characterizes $\phi$ in the following strong
sense.

\begin{theorem}\label{t.germ} Let $Y\subset \BP^n, n \geq 3,$ be a smooth
 hypersurface  of degree $2m, 2
\leq m \leq n-1$, and let $\phi: X \to \BP^n$ be a double cover
branched along $Y$. Let $U \subset X$ be a  neighborhood (in classical topology) of a
general point $x \in X$ and $\varphi: U \to \BP^n$ be a
biholomorphic immersion such that for any member $C$ of $\sK_y, y
\in U$, the image $\varphi(C \cap U)$ is contained in a line in
$\BP^n$. Then there exists a projective transformation $\psi:
\BP^n \to \BP^n$ such that $\varphi= \psi \circ (\phi|_U)$.
\end{theorem}

As a consequence, we obtain the following algebraic version.

\begin{corollary}\label{c.finite} In the setting of Theorem \ref{t.germ},
let $\hat{X}$ be an $n$-dimensional projective variety equipped
with generically finite surjective morphisms  $g: \hat{X} \to X$
and $h: \hat{X} \to \BP^n$ such that for a minimal rational curve
$C$ through a general point of $X$, there exists an irreducible component $C'$ of
$g^{-1}(C)$ whose image $h(C') \subset \BP^n$ is a line. Then there
exists an automorphism $\psi: \BP^n \to \BP^n$ such that $h= \psi
\circ \phi \circ g$.
\end{corollary}

 This is a remarkable property of the double
cover $\phi:X \to \BP^n$, because an analogous statement fails
drastically for many examples of  Fano manifolds of Picard number
1 as the following example shows.

\begin{example}\label{e.counter} Let $X \subset \BP^N$ be a Fano manifold embedded in
projective space of dimension $N > n = \dim X$ such that lines of
$\BP^N$ lying on $X$ cover $X$, i.e, minimal rational curves on
$X$ are lines of $\BP^N$ lying on $X$. There are many such
examples, e.g.,  rational homogeneous spaces under a minimal
embedding or complete intersections of low degree in $\BP^N$. We
can define a finite projection $\phi: X \to \BP^n$ to a linear
subspace $\BP^n \subset \BP^N$ by choosing a suitable linear
subspace  $\BP^{N-n-1} \subset (\BP^N \setminus X).$ Then
minimal rational curves are sent to lines in $\BP^n$ by $\phi$.
There are many different ways to choose such $\BP^{N-n-1}$ and projections. For most examples of $X$,
different choices of $\phi$ need not be related by projective
transformations of $\BP^n$.
\end{example}

What makes the difference between Corollary \ref{c.finite} and
Example \ref{e.counter}? The key point is that the ideal defining
the VMRT of the double cover, as described in Theorem
 \ref{t.VMRT}, does not contain a quadratic
polynomial.  In fact, we will prove a general version,
 Theorem \ref{t.general}, of Theorem \ref{t.germ} where the double cover $X$ is replaced by  any Fano manifold whose VMRT at a general
 point  is not contained in a
hyperquadric. In this regard, we should mention that our double
cover $\phi:X \to \BP^n$ is the {\em first known example} of a Fano
manifold with Picard number 1 whose VMRT at a general point is
not contained in a hyperquadric. Note that the VMRT's of Fano
manifolds in Example \ref{e.counter} are contained in
hyperquadrics coming from the second fundamental form of $X
\subset \BP^N$.

\medskip
   Theorem \ref{t.germ} has an application in the study of  morphisms between
  double covers.  There
 have been several works, e.g., \cite{A1}, \cite{A2}, \cite{HM03},
 \cite{IS}
and \cite{S}, classifying finite morphisms between Fano threefolds
of Picard number 1. But there still remain a few unsettled cases.
One such case is finite morphisms between double covers of $\BP^3$
branched along smooth quartic surfaces. When the quartic surfaces
do not contain lines, Theorem 1.5 (more precisely, in the
subsection (4.2.2)) of \cite{S} proves that such a morphism must
be an isomorphism.  However, the approach of \cite{S} has
technical difficulties when the quartic surfaces contain lines.
Using Theorem \ref{t.germ}, we can settle this case.  More
precisely, we obtain the following.

\begin{theorem}\label{t.map}
 Let $Y_1, Y_2 \subset \BP^n, n \geq 3,$ be two smooth hypersurfaces  of degree
$2n-2$. Let $\phi_1: X_1 \to \BP^n$ (resp. $\phi_2: X_2 \to
\BP^n$)  be a double cover of $\BP^n$ branched along  $Y_1$ (resp.
$Y_2$). Suppose there exists a finite morphism $f:X_1 \to X_2$.
Then $f$ is an isomorphism. \end{theorem}

\medskip
Another application of Theorem \ref{t.germ} is  the following
problem, which is Problem 7.9 in \cite{Hw12}.

\begin{problem}[Liouville-type extension problem]\label{q.Liouville} Let $X$ be a Fano manifold of Picard
number 1. Let $U_1$ and $U_2$ be two connected open subsets (in classical topology) in
$X$. Suppose that we are given  a biholomorphic map $\gamma: U_1
\to U_2$ such that for any minimal rational curve $C \subset X$,
there exists another minimal rational curve $C'$ with $\gamma(U_1
\cap C) = U_2 \cap C'$. Then does there exist $\Gamma \in {\rm
Aut}(X)$ with $\Gamma|_{U_1} = \gamma$? \end{problem}

Problem \ref{q.Liouville} is called  Liouville-type extension,
 because  Liouville's theorem in conformal geometry
 gives an affirmative answer to Problem
\ref{q.Liouville} when $X$ is a smooth quadric hypersurface in
$\BP^{n+1}$. An affirmative answer to Problem \ref{q.Liouville} is
known if $\dim \sK_x>0$ for a general $x \in X$ (this is
essentially proved in \cite{HM01}). However, when $\dim \sK_x =0$,
 affirmative answers are known  only in a small number of examples of
 $X$, such as hypersurfaces of degree $n$ in
 $\BP^{n+1}$ and Mukai-Umemura threefolds (see Section 7 of \cite{Hw12}).

 Theorem \ref{t.germ} enables us to give  a stronger form of
Liouville-type extension for our double cover $X$ as follows.

 \begin{theorem}\label{t.Liouville}
 Let $Y_1, Y_2 \subset \BP^n, n \geq 3,$ be two smooth hypersurfaces  of degree
$2m, 2\leq m \leq n-1$. Let $\phi_1: X_1 \to \BP^n$ (resp.
$\phi_2: X_2 \to \BP^n$)  be a double cover of $\BP^n$ branched
along  $Y_1$ (resp. $Y_2$). Let $U_1 \subset X_1$ and $U_2 \subset
X_2$ be two connected open subsets. Suppose that we are given  a
biholomorphic map $\gamma: U_1 \to U_2$ such that for any minimal
rational curve $C_1 \subset X_1$, there exists a minimal rational
curve $C_2 \subset X_2$ with $\gamma(U_1 \cap C_1) = U_2 \cap
C_2$. Then we can find a biregular morphism $\Gamma: X_1 \to X_2$
with $\Gamma|_{U_1} = \gamma$.
\end{theorem}

\bigskip
The organization of this paper is as follows. In Section 2, we
will present some basic facts concerning  double covers of $\BP^n$
and their minimal rational curves. Theorem \ref{t.VMRT} and
Theorem \ref{t.converse} will be proved in Section 3. In Section
4, the variation of VMRT is studied and Theorem \ref{t.LR} will be
proved. Finally, in Section 5, we review the notion of projective
connections to prove a general version of Theorem \ref{t.germ} and
explain how  Theorem \ref{t.map} and Theorem \ref{t.Liouville} can
be derived from Theorem \ref{t.germ}.

\section{Minimal rational curves and ECO lines}

Throughout, we will fix integers $n \geq 3$ and $2 \leq m \leq
n-1$. Let $Y\subset \mathbb P^n$ be a smooth hypersurface of
degree $2m$. Let $\phi: X\rightarrow \mathbb P^n$ be a double
cover of $\mathbb P^n$ ramified along $Y$. Such a double cover
arises as a submanifold in the line bundle $\sO_{\BP^n}(m)$ as
explained in pp.242-244 of \cite{La}. This implies the following
uniqueness result, where the slightly awkward appearance of the open
subsets $U_1$ and $U_2$ are for our later use in Section 5.

\begin{lemma}\label{l.unique}
Given a smooth hypersurface $Y \subset \BP^n$ of degree $2m$, let
$\phi_1: X_1 \to \BP^n$ and $\phi_2: X_2 \to \BP^n$ be two choices
of double covers of $\BP^n$ branched along $Y$. Let $U_1 \subset
X_1$ and $U_2 \subset X_2$ be two connected open subsets (in classical topology) with
$\phi_1(U_1) = \phi_2(U_2)$. Then there exists a biregular
morphism $\Gamma: X_1 \to X_2$ with $\Gamma(U_1) = U_2$ and
$\phi_1 =  \phi_2\circ\Gamma$.
\end{lemma}

\begin{definition}\label{d.mrc}  Let $\phi: X \to \BP^n$ be a double cover
branched along a smooth hypersurface $Y \subset \BP^n$ of degree
$2m$. A rational curve $C \subset X$ with $\phi(C) \not\subset Y$
 is a {\em minimal rational curve} if it has degree 1 with respect to $\phi^* \sO_{\BP^n}(1)$.
 For $x \in X \setminus \phi^{-1}(Y)$, we denote by $\sK_x$ the
 (normalized) space of minimal rational curves through $x$. It is
 known (e.g. II.3.11.5 in \cite{Ko}) that $\sK_x$ is a union
 of finitely many nonsingular  projective varieties.
From the adjunction formula
$$K_X=\phi^*(K_{\mathbb P^n}+\frac{1}{2}[Y])=\phi^*\mathcal O_{\mathbb
P^n}(-n+m-1),$$   $X$ is a Fano manifold of  index $n-m+1$ and
$\dim \sK_x = n-m-1$.
\end{definition}

\begin{definition}\label{d.ecoline}
Let $Y \subset \BP^n$ be an irreducible reduced hypersurface. A line
$\ell \subset \BP^n$ is an {\em ECO (Even Contact Order) line }
with respect to $Y$ if $\ell\not\subset Y$ and the local
intersection number  at each point of $\ell \cap Y$ is even. For a
point $x \in \BP^n \setminus Y$, identify the space of lines
through $x$ with the projective space $\BP T_x(\BP^n)$ and denote
by $\sE^Y_x \subset \BP T_x(\BP^n)$ the space of ECO lines through
$x$  with respect to $Y$.
\end{definition}

Next proposition is a direct generalization of well-known facts
for $(n, m) = (3,2)$ (e.g. \cite{Ti}).

\begin{proposition}\label{p.mrc} In the setting of Definition
\ref{d.mrc}, an irreducible reduced curve $C\subset X$ with
$\phi(C) \not\subset Y$ is a minimal rational curve if and only if
the image curve $\phi(C)\subset \mathbb P^n$ is an ECO line with
respect to $Y$. Moreover, a minimal rational curve $C$ is smooth
and $\phi|_C: C \to \phi(C)$ is an isomorphism.
\end{proposition}

\begin{proof}
Let $C \subset X$ be an irreducible curve such that
$\ell:=\phi(C)$ is an ECO line with respect to $Y$. Suppose that
$\phi_C: C \to \ell$ is not birational, i.e., $C =
\phi^{-1}(\phi(C))$. For a point $z\in \phi(C)\cap Y$, let $t$ be
a local uniformizing parameter on $\ell$ at $z$ and let $r_z$ be
the local intersection number of $\ell$ and $Y$ at $z$. Then $C$
is analytically defined by the equation $s^2=t^{r_z}$ (cf.
\cite{La} pp.242-244). Let $\tilde{C}$ be the normalization of
$C$. Since $r_z$ is even for any choice of $z \in \phi(C) \cap Y$,  the composition of the normalization morphism $\tilde
C\rightarrow C$ and the covering morphism $\phi|_C: C\rightarrow
\ell$ induces a morphism $ \tilde C\rightarrow \ell$ of degree 2
without ramification point, a contradiction.  Thus $\phi|_C: C \to
\ell$ is birational and $C$ has degree 1 with respect to
$\phi^*\sO_{\BP^n}(1)$.

Conversely, if $C$ is a minimal rational curve, then
$\ell:=\phi(C)$ is a line in $\mathbb P^n$ with $\ell \not\subset
Y$ and $\phi|_C: C \to \ell$ must be birational. Thus
$\phi^{-1}(\ell)$ has an irreducible component $C'$ different from
$C$ with $\phi(C \cap C') = \ell \cap Y$. By the same argument as
before, if the local intersection number $r_z$ at $z \in \ell \cap
Y$ is odd, the germ of $\phi^{-1}(\ell)$ over $z$, defined by $s^2
= t^{r_z}$, is irreducible, a contradiction. Thus $r_z$ is even
for all $z \in \ell \cap Y$ and $\ell$ is an ECO line. Moreover,
$C$ must be smooth and the morphism $\phi|_C: C \to \ell$ is an
isomorphism.
\end{proof}

We have the following consequence.

\begin{proposition}\label{p.unique}
In the setting of Proposition \ref{p.mrc}, let  $Y' \subset \BP^n$
be an irreducible reduced hypersurface distinct from $Y$. Then  a general ECO line with respect to $Y$ intersects $Y'$
transversally. In particular,  a general ECO line with respect to
$Y$ cannot be an ECO line with respect to $Y'$.
\end{proposition}

\begin{proof}
On a Fano manifold $X$,
for any subset $Z \subset X$ of codimension $\geq 2$ and any reduced hypersurface $D \subset X$,   a general minimal rational curve is disjoint from
$Z$ (e.g. Lemma 2.1 in \cite{Hw01}) and intersects $D$
  transversally (the proof is similar to the proof of Lemma 2.1 in \cite{Hw01}).
   Putting $Z= \phi^{-1}(Y \cap Y')$ and
  $D= \phi^{-1}(Y')$ for our $\phi: X \to \BP^n$ branched along $Y$,
  we see that a general minimal rational curve $C$ intersects $\phi^{-1}(Y')$ transversally and $\phi(C) \cap Y \cap Y' = \emptyset$. Thus $\phi(C)$ intersects $Y'$ transversally.
 \end{proof}

\begin{proposition}\label{p.isom}
In the setting of Proposition \ref{p.mrc}, let $x$ be a general
point of $X$. Let $\tau_x: \sK_x \to \BP T_x(X)$ be the tangent
morphism associating each member of $\sK_x$ its tangent direction at
$x$. Then $\tau_x$ is an embedding and the VMRT $\mathcal C_x= {\rm
Im}(\tau_x) \subset \BP T_x(X)$ is a nonsingular projective variety
with finitely many components of dimension $n-m-1$, isomorphic to
$\sE_{\phi(x)} \subset \BP T_{\phi(x)}(\BP^n)$.
\end{proposition}

\begin{proof}
The differential $d \phi_x: \BP T_x(X) \to \BP T_{\phi(x)}(\BP^n)$
sends $\sC_x \subset \BP T_x(X)$  isomorphically to
$\sE^Y_{\phi(x)} \subset \BP T_{\phi(x)}(\BP^n)$ by Proposition
\ref{p.mrc}. It follows that $\tau_x$ is injective because lines
on $\BP^n$ are determined by their tangent directions.

Since we know that $\sK_x$ is nonsingular of dimension $n-m-1$, to
prove that $\tau_x$ is an embedding, it remains to show that
$\tau_x$ is an immersion. By Proposition 1.4 in \cite{Hw01}, this
is equivalent to showing that for any member $C\subset X$ of
$\sK_x$, the normal bundle $N_{C/X}$ satisfies
$$N_{C/X}= O_{\mathbb P^1}(1)^{
n-m-1}\oplus \mathcal O_{\mathbb P^1}^{m}.$$ By the
generality of $x$,  we can write
$$N_{C/X}= \mathcal O_{\mathbb P^1}(a_1)\oplus \cdots \oplus
\mathcal O_{\mathbb P^1}(a_{n-1})$$ for  integers
$a_1\geq \cdots \geq a_{n-1}\geq 0$ satisfying $\sum_{i} a_i = n-m-1$.
 Since $\phi$ is unramified at general points of $C$ and $\phi|_C: C \to \ell := \phi(C)$
 is an isomorphism,
 we have an injective sheaf homomorphism
$$\phi_*: N_{C/X} \to N_{\ell/\BP^n} =\mathcal O(1)^{n-1}.$$ Thus $a_1 \leq 1$.
   It follows
that $a_1 = \cdots = a_{n-m-1} =1$ and $a_{n-m} = \cdots = a_{n-1}
=0$. \end{proof}

\section{Defining equations of VMRT}\label{s.ECO}

\begin{definition}\label{d.weight} A polynomial $A(t_1, \ldots, t_m)$ in $m$ variables
is said to be {\em weighted homogeneous of weighted degree $k$} if
it is of the form
$$A(t_1,...,t_m)=\sum_{1\cdot i_1+\cdots +m\cdot
i_m=k}c_{i_1,...,i_m}t_1^{i_1}\cdots t_m^{i_m}$$ with coefficients
$c_{i_1,...,i_m}\in \mathbb C$. An equivalent way of defining it
is as follows. We define the {\em weighted degree} of each
variable $t_i$ by ${\rm wt}(t_i) := i$ and each monomial by ${\rm
wt}(t_{i_1} \cdots t_{t_N}) := \sum_{j=1}^N {\rm wt}(t_{i_j})$.
Then $A$ is weighted homogeneous of weighted degree $k$ if all
monomial terms in $A$ have weighted degree $k$.
\end{definition}

\begin{definition}\label{d.ecopoly}
A polynomial  of degree $2m, m\geq 1,$ in one variable with
complex coefficients is an {\em ECO polynomial} if it can be
written as the square of a polynomial of degree $m$.
\end{definition}

\begin{proposition}\label{p.ecopoly}
For any positive integer $m$, there exists a weighted homogeneous
polynomial $A_{k}(t_1, \ldots, t_m)$ of weighted degree $k$ for each $k,
m+1 \leq k \leq 2m$,
 such that a polynomial in one variable $\lambda$ of degree $2m$
$$a_{2m} \lambda^{2m} + a_{2m-1} \lambda^{2m-1} + \cdots + a_1 \lambda +1$$ is an ECO polynomial
if and only if $ a_k = A_k(a_1, \ldots,a_m) $ for each  $m+1 \leq
k \leq 2m.$ \end{proposition}

\begin{remark}  Our proof below gives a recursive formula for
$A_k$, but an explicit expression of the polynomials $A_k$ will
not be needed in this paper.
\end{remark}

\begin{proof}
Suppose that $$a_{2m} \lambda^{2m} + a_{2m-1} \lambda^{2m-1} +
\cdots + a_1 \lambda +1$$ is an ECO polynomial. We can find
$(\sigma_1, \ldots, \sigma_m) \in \C^m$ such that $$a_{2m}
\lambda^{2m} + a_{2m-1} \lambda^{2m-1} + \cdots + a_1 \lambda +1 =
(\sigma_m \lambda^m + \sigma_{m-1} \lambda^{m-1} + \cdots +
\sigma_1 \lambda+1)^2.$$ For convenience,  define $\sigma_0=1,
\sigma_{m+1} = \cdots = \sigma_{2m}=0,$ so that we can write, for
each $k, 1 \leq k \leq 2m$,
$$a_k = \sum_{i=0}^k \sigma_i \sigma_{k-i}. $$
Using
$${a}_k=\sum_{i=0}^{k}\sigma_i\sigma_{k-i}=\sum_{i=1}^{k-1}\sigma_i\sigma_{k-i}+2\sigma_k,$$
we have
$$\sigma_k=\frac{{a}_k-\sum_{i=1}^{k-1}\sigma_i\sigma_{k-i}}{2} $$  for
$k=1,2,...,m$. Thus
\begin{align*}
 \sigma_1=\frac{{a}_1}{2}, \
 \sigma_2=\frac{{a}_2-\sigma_1^2}{2}=\frac{{a}_2}{2}-\frac{{a}_1^2}{8}, \ \cdots.
\end{align*}
Using induction on $k$, we see that $$ \sigma_k = G_k(a_1, \ldots,
a_m) \mbox{ for each } k, 1 \leq k \leq m,$$ where $G_k(t_1,
\ldots, t_m)$ is a weighted homogeneous polynomial of weighted
degree $k$.
 Setting
$G_0=1, G_{m+1} = \cdots = G_{2m} =0$, we see that
$$a_{k} = \sum_{\ell=0}^{k} G_{\ell}(a_1, \ldots, a_m) G_{k-\ell}(a_1,
\ldots, a_m)$$ for all $m+1\leq k \leq 2m$. Define
$$A_{k}(t_1, \ldots, t_m):= \sum_{\ell=0}^{k} G_{\ell}(t_1, \ldots,
t_m) G_{k-\ell}(t_1, \ldots, t_m).$$ Then $A_k$ is a weighted
homogeneous polynomial of weighted degree $k$ such that $a_k=
A_k(a_1, \ldots, a_m)$ for each $m+1 \leq k \leq 2m$.

Conversely, given any $(a_1, \ldots, a_m) \in \C^m$, let $$
a_{m+i} = A_{m+i}(a_1, \ldots,a_m) \mbox{ for each } 1 \leq i \leq
m$$ where $A_{m+i}$ is defined above.  Then for $\sigma_i =
G_j(a_1, \ldots, a_m)$, we see that
$$(\sigma_m \lambda^m + \cdots + \sigma_1 \lambda + 1)^2 = a_{2m} \lambda^{2m} +
a_{2m-1} \lambda^{2m-1} + \cdots + a_1 \lambda + 1$$ and
$$a_{2m}\lambda^{2m} + a_{2m-1} \lambda^{2m-1} + \cdots + a_1 \lambda
+1$$ is an ECO polynomial.
\end{proof}

\begin{corollary}\label{c.ecopoly}
Regard the affine space $$\mathbb A^{2m}:= \{(a_{2m}, a_{2m-1},
\ldots, a_1)\ | \ a_i \in \C \}$$ as the set of polynomials
$$ a_{2m} \lambda^{2m} + a_{2m-1} \lambda^{2m-1} + \cdots + a_1 \lambda + 1$$ of
degree $2m$ with the constant term 1. Then the set $\sD \subset
\mathbb A^{2m}$ of ECO-polynomials is a smooth complete
intersection of $m$ divisors $D_1, \ldots, D_m$ where $D_j$ is the
smooth divisor defined by $a_{m+j} = A_{m+j}(a_1, \ldots, a_m)$
where $A_{m+j}$ is the weighted homogeneous polynomial of weighted
degree $m+j$ defined in Proposition \ref{p.ecopoly}.
\end{corollary}

Using Corollary \ref{c.ecopoly}, we will study the space of ECO
lines defined in Definition \ref{d.ecoline}. For our computation, we introduce the following notation.

\begin{notation}\label{n.coordi} Choose a homogeneous coordinate system $t_0$,...,$t_n$ on
$\mathbb P^n$. Let $\BP^{n-1}_{\infty} \subset \BP^n$ be the
hyperplane defined by $t_0=0$. The restriction of $t_1, \ldots,
t_n$ on $\BP^{n-1}_{\infty}$ will be denoted by $z_1, \ldots,
z_n$. They provide a homogeneous coordinate system on
$\BP^{n-1}_{\infty}.$ Define the projective isomorphism
$\upsilon_y: \BP^{n-1}_{\infty} \to \BP T_y(\BP^n)$  at each point
$y=[1:y_1: \cdots : y_n] \in \BP^n \setminus \BP^{n-1}_{\infty}$
by sending $[z_1:\cdots:z_n]\in \BP^{n-1}_{\infty}$ to the tangent
direction of the line $$\{(y_1 + \lambda z_1, \ldots, y_n +
\lambda z_n)\ | \  \lambda \in \C\}$$ at the point $y$. The
collection $\{ \upsilon_y^{-1} \ | \ y \in \BP^n \setminus
\BP^{n-1}_{\infty}\}$ determines a canonical trivialization of the
projectivized tangent bundle
$$\upsilon^{-1}: \BP T(\BP^n \setminus \BP^{n-1}_{\infty}) \cong (\BP^n \setminus \BP^{n-1}_{\infty})
\times \BP^{n-1}_{\infty}.$$ \end{notation}

\begin{definition}\label{d.f}
For a homogeneous polynomial $f(t_0, \ldots, t_n)$ of degree $2m$,
$2 \leq m \leq n-1$ and for each integer $k, 0 \leq k \leq 2m,$
define $a^f_k(y;z) = a^f_k(y_1,\ldots,y_n;z_1, \ldots, z_n)$ to be
the polynomial in $2n$ variables satisfying
$$f(1, y_1 + \lambda z_1, \ldots, y_n + \lambda z_n) = a^f_0(y;z)
+ a^f_1(y;z) \lambda + \cdots + a^f_{2m}(y;z) \lambda^{2m}.$$ Note
that for a fixed $y$, $a^f_k(y;z)$ is a homogeneous polynomial in
$z$ of degree $k$. In particular, $a^f_0(y;z) = f(1,y_1, \ldots,
y_n)$ is independent of $z$.
\end{definition}

\begin{proposition}\label{p.maineco}
In Notation \ref{n.coordi} and Definition \ref{d.f}, let $Y
\subset \BP^n$ be the hypersurface defined by
$f(t_0,\ldots,t_n)=0$. For any point  $y \in \BP^n \setminus (Y
\cup \BP^{n-1}_{\infty}),$ the variety $\upsilon_y^{-1}(\sE^Y_y)
\subset \BP^{n-1}_{\infty}$ is (set-theoretically) the common zero
set of the homogeneous polynomials in $z$, $B^f_{k}(y;z),m+1 \leq
k \leq 2m$, defined by
$$B^f_k(y;z)= B^f_k(y_1,...,y_n;z_1,...,z_n):=\frac{a^f_k(y;z)}{a^f_0(y;z)}-A_k
\left (
\frac{a^f_1(y;z)}{a^f_0(y;z)},...,\frac{a^f_m(y;z)}{a^f_0(y;z)}\right
),$$  where $A_k$ is as in Proposition \ref{p.ecopoly}. Note that
$B^f_k(y;z)$ is homogeneous  in $z$ of degree $k$ because
$a^f_k(y;z)$ is homogeneous in $z$ of degree $k$    and $A_k$ is
weighted homogeneous of weighted degree $k$. In particular, if
$\sE^Y_y$ is of pure dimension $n-m-1$, then it is
set-theoretically a complete intersection of multi-degree $(m+1,
m+2, \ldots, 2m)$. \end{proposition}

\begin{proof}
A point $[z_1:\cdots:z_n] \in \BP^{n-1}_{\infty}$ belongs to
$\upsilon_y^{-1}(\sE^Y_y)$ if and only if the polynomial in
$\lambda$ $$f(1, y_1 + \lambda z_1, \ldots, y_n + \lambda z_n) =
a^f_0(y;z) + a^f_1(y;z) \lambda + \cdots + a^f_{2m}(y;z)
\lambda^{2m}$$ is an ECO polynomial. By Corollary \ref{c.ecopoly},
we see that $\upsilon_y^{-1}(\sE^Y_y)$ is the common zero set of
$B^f_k(y;z), m+1 \leq k \leq 2m$. \end{proof}

\begin{proposition}\label{p.converse}  Given
a general smooth complete intersection $Z \subset \BP^{n-1}$ of
multi-degree $(m+1, \ldots, 2m),$ there exist a smooth
hypersurface $Y \subset \BP^n$ of degree $2m$ and a point $y \in
\BP^n \setminus Y$, such that $\mathcal E_y^Y \subset \mathbb P
T_y(\BP^n)$ is projectively equivalent to $Z \subset \BP^{n-1}$.
In particular, for a general hypersurface $Y \subset \BP^n$ of
degree $2m$ and a general $y\in \BP^n \setminus Y$, the variety of
ECO lines $\mathcal E^Y_y\subset \mathbb P T_y(\BP^n)$ is a smooth
complete intersection of degree $(m+1,...,2m)$.
\end{proposition}

\begin{proof}
Denote by $\{b_{k}(z_1, \ldots, z_n)\ | \  m+1 \leq k \leq 2m\}$
homogeneous polynomials with $\deg b_k = k$ defining $Z$. By the
generality of $Z$, we may assume that \begin{itemize} \item[(i)]
the affine hypersurface
$$1 + b_{m+1}(t_1, \ldots, t_n) + \cdots + b_{2m}(t_1, \ldots, t_n)
=0$$ in $\C^n = \{ (t_1, \ldots, t_n), t_i \in \C \}$ is smooth
and \item[(ii)] the projective hypersurface $$b_{2m}(z_1, \ldots,
z_n) =0$$ in $\BP^{n-1}$ with homogeneous coordinates $(z_1:z_2:
\cdots : z_n)$ is smooth.
\end{itemize}  Let $Y \subset \BP^n$ be the hypersurface of degree
$2m$ defined by the polynomial
$$f(t_0, t_1, \ldots, t_n) := t_0^{2m} + t_0^{m-1} b_{m+1}(t_1, \ldots, t_n) +
\cdots + t_0 b_{2m-1}(t_1, \ldots, t_n) + b_{2m}(t_1, \ldots,
t_n).$$ Then $Y$ is a smooth hypersurface because it has no
singular point on its intersection with  the hyperplane $t_0 =0$
by the assumption (ii), while it has no singular point on the
affine space $t_0 \neq 0$ by the assumption (i). In Notation
\ref{n.coordi}, consider the point $y=[1:0:0:\cdots:0] \in \BP^n
\setminus (\BP^{n-1}_{\infty}\cup Y)$. Then
$$f(1, y_1 + \lambda z_1, \ldots, y_n +\lambda z_n) = f(1, \lambda
z_1, \ldots, \lambda z_n) = 1+ \lambda^{m+1} b_{m+1}(z) + \cdots +
\lambda^{2m} b_{2m}(z).$$ Comparing with Definition \ref{d.f}, we
obtain
$$a^f_0(y;z)  =1, \; a_1^f(y;z)= \cdots =
a^f_m(y;z) =0, \; a^f_k(y;z) = b_k(z) \mbox{ for } m+1 \leq k \leq
2m.$$ In the notation of Proposition \ref{p.maineco},
$$B^f_k(y; z_1, \ldots, z_n) = a^f_k(z_1, \ldots, z_n) = b_k(z_1, \ldots,
z_n)$$ for $m+1 \leq k \leq 2m.$ This implies that $\sE^Y_y, y:=
[1:0:\cdots:0]$, is projectively equivalent to $Z \subset
\BP^{n-1}$.
\end{proof}

\begin{proof}[Proof of Theorem \ref{t.VMRT} and Theorem \ref{t.converse}]

Theorem \ref{t.converse} is a direct consequence of Proposition
\ref{p.isom} and Proposition \ref{p.converse}.

To prove Theorem \ref{t.VMRT}, it suffices to show by Proposition
\ref{p.isom} that for any smooth $Y \subset \BP^n$ and a general
point $x \in \BP^n \setminus Y$, the subvariety $\sE^Y_x \subset
\BP T_x(\BP^n)$ is a smooth complete intersection of multi-degree
$(m+1, \ldots, 2m)$. Proposition \ref{p.converse} says that this
is O.K. if $Y$ is a general hypersurface.

To check it for any smooth $Y \subset \BP^n$, choose a deformation
$\{ Y_t \ | \; |t|< \epsilon \}$ of $Y=Y_0$ such that for a
Zariski open subset $U_t \subset \BP^n \setminus Y_t$,
$\sE^{Y_t}_x \subset \BP T_{x}(\BP^n)$ is a smooth complete
intersection for any $t \neq 0$ and any $x \in U_t$. By shrinking
$\epsilon$ if necessary, the intersection $\cap_{t\neq 0} U_t$ is
nonempty.  By Proposition \ref{p.isom}, we have a Zariski open
subset $U \subset \BP^n \setminus Y$ such that $\sE^Y_x$ is smooth
for any $x \in U$. Pick a point $x \in  (\cap_{t \neq 0} U_t) \cap
U $.

We can construct a smooth family $\{\phi_t:X_t \to \BP^n\ | \; |t|
< \epsilon\}$ of  double covers of $\BP^n$ branched along $Y_t$'s.
Choose $z_t \in  \phi_t^{-1}(x)$ in a continuous way.  The proof
(e.g. II.3.11.5 in \cite{Ko}) of the smoothness of $\sK_x$
mentioned in Definition \ref{d.mrc} works for the family
$\sK_{z_t}$, i.e., the family $\{ \sK_{z_t}\ | \; |t| <\epsilon\}$
is a flat family of nonsingular projective subvarieties. Via
Proposition \ref{p.isom}, this implies that $\{ \sE^{Y_t}_x\ | \;
|t|<\epsilon\}$ is a flat family of nonsingular projective
varieties in $\BP T_x(\BP^n)$.
 By our
choice of $x$, $\sE^{Y_t}_x$ is (scheme-theoretically) a smooth
complete intersection for $t \neq 0$, while $\sE^{Y_0}_x$ is a
nonsingular variety which is set-theoretically a complete
intersection of the same multi-degree as $\sE^{Y_t}_x, t\neq 0,$
by Proposition \ref{p.maineco}.  We conclude that $\sE^{Y_0}_x$ is
also a smooth complete intersection of multi-degree $(m+1, \ldots,
2m)$.
\end{proof}

\section{Variation of VMRT}

\begin{notation}\label{n.V} Let $V_k$ be the vector space of homogeneous polynomials of
degree $k$ in $z_1$,...,$z_n$. Each polynomial $h\in V_k$ is of
the form
$$h(z_1,...,z_n)=\sum_{i_1+\cdots+i_n=k}e_{i_1,...,i_n}z_1^{i_1}\cdots
z_n^{i_n}.$$ Regarding $V_k$ as a complex manifold, take
$\{e_{i_1,..,i_n}\}_{i_1+\cdots+i_n=k}$ as linear coordinates on
$V_k$ and $$\left\{ \frac{\partial}{\partial e_{i_1,...,i_n}} \
\Bigr| \ i_1+\cdots+i_n=k \right\}$$ as a basis for the tangent
spaces $T_{h}(V_k)$ of $V_k$ at each $h\in V_k$. There is a
canonical isomorphism between $V_k$ and $T_{h}(V_k)$ identifying a
polynomial
$$\sum_{i_1+\cdots+i_n=k}E_{i_1,...,i_n}z_1^{i_1}\cdots z_n^{i_n}\in
V_k$$ with the tangent vector
$$\sum_{i_1+\cdots+i_n=k}E_{i_1,...,i_n}\frac{\partial}{\partial
e_{i_1,...,i_n}}\in T_h(V_k).$$ \end{notation}

\begin{notation}\label{n.f_k}
For a homogeneous polynomial $f(t_0, \ldots, t_n)$ of degree $2m,
2 \leq m \leq n-1$,  write
$$f(t_0,...,t_n)=t_0^{2m}f_0(t_1,...,t_n)+t_0^{2m-1}f_1(t_1,...,t_n)+\cdots+f_{2m}(t_1,...,t_n)$$
where $f_k(t_1,...,t_n)$ is a homogeneous polynomial of degree
$k=0,...,2m$ in $t_1$,...,$t_n$. Comparing with Definition
\ref{d.f}, we have $$f(1,0,...,0)=f_0(z_1,...,z_n) = a^f_0(0;z)
\mbox{ and } a^f_i(0;z)=f_i(z) \mbox{ for } i=1,..., 2m.$$
\end{notation}

\begin{definition}\label{d.phi}
For a homogeneous polynomial $f(t_0, \ldots, t_n)$ of degree $2m,
2 \leq m \leq n-1$, let $Y \subset \BP^n$ be the hypersurface
defined by $f(t_0,\ldots,t_n)=0$ and define  a morphism $$\mu:
\BP^n \setminus (\BP^{n-1}_{\infty}\cup Y) \to V_{m+1}$$ by
sending $y=[1:y_1: \cdots :y_n]$  to the polynomial in $z$
$$\mu(y):= [B^f_{m+1}(y;z)] \in V_{m+1}$$ with $B^f_{m+1}(y;z)$ as
in Proposition \ref{p.maineco}.
\end{definition}

\begin{proposition}\label{p.dphi}
In Notation \ref{n.f_k} and  Definition \ref{d.phi},
 assume that
$$f(1,0,...,0)=f_0(z)=1 \mbox{ and } f_1(z)=\cdots=f_m(z)=0.$$  Then for
$x=[1:0:\cdots:0] \in \BP^n \setminus (\BP^{n-1}_{\infty}\cup Y)$
and $\sum_{i=1}^n v_i \frac{\partial}{\partial y_i} \in
T_x(\BP^n)$,
$$d \mu_x(\sum_{i=1}^n v_i \frac{\partial}{\partial y_i}) =
[\sum_{i=1}^n v_i \frac{\partial f_{m+2}}{\partial t_i}(z)]  \in
T_{\mu(x)}(V_{m+1}) = V_{m+1}.$$ \end{proposition}

We will use the following lemma.

\begin{lemma}\label{l.Tim}
In Notation \ref{n.f_k}, set $f_{2m+1} =0$ for convenience. Assume
that  $$f(1,0,...,0)=f_0(z)=1.$$ Then
 $B^f_{k}(y;z)$ of Proposition \ref{p.maineco} satisfies
  \begin{eqnarray*} \frac{\partial
B^f_{k}(y;z)}{\partial y_i}\Bigr|_{(0;z)} &=&  \frac{\partial
f_{k+1}}{\partial t_i}(z)-f_k(z)\frac{\partial f_1 }{\partial
t_i}(z) \\ & & -\sum_{j=1}^{m}\frac{\partial A_{k}}{\partial
x_j}(f_1(z),...,f_m(z))\left( \frac{\partial f_{j+1}}{\partial
t_i}(z)-f_j(z)\frac{\partial f_1 }{\partial t_i}(z)\right)
\end{eqnarray*} for all $k=m+1,...,2m$ and $i=1,...,n$.
\end{lemma}
\begin{proof}
 In the equality $$ \frac{\partial f(t)}{\partial
t_i}\Bigr|_{t=(1,\lambda z_1,...,\lambda z_n)} =\frac{\partial
f(1,y_1+\lambda z_1,...,y_n+\lambda z_n)}{\partial
y_i}\Bigr|_{y_1=\cdots=y_n=0},$$ the left hand side can be
written,  via  Notation \ref{n.f_k}, $$\frac{\partial
f_1}{\partial t_i}(z_1,...,z_n)+ \frac{\partial f_2}{\partial
t_i}(z_1,...,z_n)\lambda+\cdots+\frac{\partial f_{2m+1}}{\partial
t_i}(z_1,...,z_n)\lambda^{2m}.$$ On the other hand, the right hand
side is, by Definition \ref{d.f}, $$\frac{\partial a^f_0}{\partial
y_i}(0;z)+\frac{\partial a^f_1}{\partial
y_i}(0;z)\lambda+\cdots+\frac{\partial a^f_{2m}}{\partial
y_i}(0;z)\lambda^{2m}.$$ Therefore for each $i=1, \ldots, n$,
$$\frac{\partial a^f_{2m}}{\partial y_i}(0;z) =0 \mbox{ and }
\frac{\partial a^f_k}{\partial y_i}(0;z)=\frac{\partial
f_{k+1}}{\partial t_i}(z)\text{ for } k=0,...,2m.$$
 From this and
the assumption that $a^f_0(0;z)=f(1,0,...,0)=1$, we obtain
\begin{eqnarray*} \frac{\partial}{\partial y_i}\left
(\frac{a^f_k(y;z)}{a^f_0(y;z)}\right )\Bigr|_{(y;z)=(0;z)}&=&
\frac{a^f_0(0;z)\frac{\partial a^f_k}{\partial
y_i}(0;z)-a^f_k(0;z)\frac{\partial a^f_0}{\partial
y_i}(0;z)}{a^f_0(0;z)^2} \\ &=& \frac{\partial f_{k+1}}{\partial
t_i}(z)-f_k(z)\frac{\partial f_1 }{\partial t_i}(z)
\end{eqnarray*} for all
$k=0,...,2m$ and $i=1,...,n$. Thus
\begin{eqnarray*}\frac{\partial B^f_{k}(y;z)}{\partial
y_i} \Bigr|_{(0;z)} &=& \frac{\partial}{\partial y_i}\left
(\frac{a^f_{k}(y;z)}{a^f_0(y;z)}\right )\Bigr|_{(0;z)}\\
& & -\sum_{j=1}^{m}\frac{\partial A_{k}}{\partial
x_j}\left(\frac{a^f_1(0;z)}{a^f_0(0;z)},...,\frac{a^f_m(0;z)}{a^f_0(0;z)}\right)\frac{\partial}{\partial
y_i}\left
(\frac{a^f_{j}(y;z)}{a^f_0(y;z)}\right )\Bigr|_{(0;z)}\\
&=& \frac{\partial f_{k+1}}{\partial t_i}(z)-f_k(z)\frac{\partial
f_1 }{\partial t_i}(z) \\ & & -\sum_{j=1}^{m}\frac{\partial
A_{k}}{\partial x_j}(f_1(z),...,f_m(z))\left(\frac{\partial
f_{j+1}}{\partial t_i}(z)-f_j(z)\frac{\partial f_1 }{\partial
t_i}(z)\right)\end{eqnarray*} for all $k=m+1,...,2m$ and
$i=1,...,n$.
\end{proof}

\begin{proof}[Proof of Proposition \ref{p.dphi}]
Since $A_k(x_1,...,x_m)$ is weighted homogeneous of weighted
degree $k$, if $k \geq m+1$, then the linear part of
$A_k(x_1,...,x_m)$ does not contain variables $x_1$,...,$x_m$.
Therefore for all $k=m+1,...,2m$ and $j=1,...,m,$
$$\frac{\partial A_k}{\partial
x_j}\Bigr|_{(0,...,0)}=0.$$  Thus putting $f_1= \cdots = f_m=0$ in
Lemma \ref{l.Tim}, we obtain
$$\frac{\partial B^f_{m+1}}{\partial y_i}(0;z)=\frac{\partial
f_{m+2}}{\partial t_i}(z).$$ It follows that
$$d
\mu_x(\sum_{i=1}^n v_i \frac{\partial}{\partial y_i})
=\sum_{i=1}^n v_i\frac{\partial B^f_{m+1}}{\partial y_i}(0;z) =
\sum_{i=1}^n v_i\frac{\partial f_{m+2}}{\partial t_i}(z).$$
\end{proof}

\begin{notation}\label{n.GL}
Denote  the action of $A \in GL(n,\mathbb C)$ on $\C^n$  by $(z_1,
\ldots, z_n) \mapsto A(z_1, \ldots, z_n).$  We have the natural
induced action on $V_k$ given by
$$(A. h)(z_1,...,z_n):=h(A^{-1}(z_1,...,z_n)),\ h\in V_k. $$
Denote the orbit of $h\in V_k$ by
$$GL(n,\mathbb C).h:=\{A.h\bigm|\ A\in GL(n,\mathbb C)\}\subset V_k.$$
\end{notation}

\begin{proposition}\label{p.Torbit}
We use the terminology of Notation \ref{n.V} and Notation
\ref{n.GL}. A tangent vector
$$\sum_{i_1+\cdots+i_n=k}E_{i_1,...,i_n} \frac{\partial}{\partial
e_{i_1,...,i_n}}  \in T_h(V_k)$$ is tangent to the orbit
$GL(n,\mathbb C).h$ if and only if there exists an $(n\times n)$
matrix $(s_{j}^i)_{i,j=1,...,n}$ such that
$$\sum_{i_1+\cdots+i_n=k}E_{i_1,...,i_n}z_1^{i_1}\cdots z_n^{i_n}=\frac{d}{dt} h(z_1+t\sum_{i=1}^n
s_1^iz_i,...,z_n+t\sum_{i=1}^n s_n^iz_i)\Bigr|_{t=0}.$$
\end{proposition}
\begin{proof}
Define a morphism $$\alpha_h:GL(n,\mathbb C)\rightarrow V_k$$
sending $A$ to $A.h$. Then $GL(n,\mathbb C).h$ is the image of
$\alpha_h$ and $\alpha_h(I)=h$ where $I$ is the $(n\times n)$
identity matrix. The tangent space of $GL(n,\mathbb C).h$ at $h$
is the image of the differential
$$d(\alpha_{h})_{I}:T_{I}(GL(n,\mathbb C))\rightarrow T_h(V_k).$$
Let us identify $T_{I}(GL(n,\mathbb C))$ with the vector space
$M_n$ of all $(n\times n)$ matrices so that $A\in M_n$ corresponds
to the tangent vector at $I$ of the curve $c(t)=I+tA$ which is
indeed a curve on $GL(n,\mathbb C)$ for sufficiently small $t$.
Since $\alpha_h\circ c(t)$ is the polynomial $h((I+tA)^{-1}(z_1,
\ldots, z_n))$,  the differential $d(\alpha_h)_I$ sends
$A=\frac{d}{dt} c(t)\Bigr|_{t=0}$ to $\frac{d}{dt}
h((I+tA)^{-1}(z_1, \ldots, z_n))\Bigr|_{t=0}$ which is of the form
on the right hand side of the equation in the proposition.
\end{proof}

\begin{proposition}\label{p.ImOb}
In the setting of Proposition \ref{p.dphi},
$$ T_{\mu(x)}(GL(n,\mathbb
C).\mu(x))=\left\{\sum_{i,j=1}^n s_j^iz_i \frac{\partial
f_{m+1}}{\partial t_j}(z) \ \Bigr| \  s^i_j\in\mathbb
C\right\}\subset T_{\mu(x)}(V_{m+1}) = V_{m+1}.$$\end{proposition}

\begin{proof} From
$\mu(x)=[B_{m+1}(0;z)] \in V_{m+1}$ and Proposition
\ref{p.Torbit}, we get
\begin{equation*}\label{eq.1}
T_{\mu(x)}(GL(n,\mathbb C).\mu(x))=\left\{ \frac{d}{dt}
B^f_{m+1}(0; z_1+t\sum_{i=1}^n s_1^iz_i,...,z_n+t\sum_{i=1}^n
s_n^iz_i)\Bigr|_{t=0} \;  \Bigr| \ s^i_j\in\mathbb
C\right\}.\end{equation*} From Notation \ref{n.f_k},  we have
$a^f_0(0;z)=f(1,0,...,0)=1$ and $a^f_i(0;z)=f_i(z)=0$ for
$i=1,...,m$. Thus
$$B^f_{m+1}(0;z)=\frac{a^f_{m+1}(0;z)}{a^f_0(0;z)}-A_{m+1} \left
 ( \frac{a^f_1(0;z)}{a^f_0(0;z)},...,\frac{a^f_m(0;z)}{a^f_0(0;z)}\right )=f_{m+1}(z).$$
So the following equalities hold
\begin{align*}
&\frac{d}{dt} B^f_{m+1}(0;z_1+t\sum_{i=1}^n s_1^iz_i,...,z_n+t\sum_{i=1}^n s_n^iz_i)\Bigr|_{t=0}\\
&=\frac{d}{dt} f_{m+1}(z_1+t\sum_{i=1}^n
s_1^iz_i,...,z_n+t\sum_{i=1}^n s_n^iz_i)\Bigr|_{t=0}\\
&=\sum_{i,j=1}^{n}s^i_jz_i\frac{\partial f_{m+1}}{\partial t_j}(z).
\end{align*} Putting it in the above expression for $T_{\mu(x)}(GL(n,\mathbb
C).\mu(x))$, we obtain the result.
\end{proof}

\begin{proposition}\label{p.trans}
There exists a smooth hypersurface $Y\subset \BP^n, n \geq 4,$
defined by a homogeneous polynomial $f$ of degree $2m, 2 \leq m
\leq n-1,$ such that, for  a general $x\in \BP^n \setminus
(\BP^{n-1}_{\infty} \cup Y)$, using the terminology of Definition
\ref{d.phi},
$${\rm rank}(d\mu_x)=n,\; \dim_{\mathbb C} T_{\mu(x)}(GL(n,\mathbb
C).\mu(x))=n^2\; \text{ and } $$ $${\rm Im}(d\mu_{x})\cap
T_{\mu(x)}(GL(n,\mathbb C).\mu(x))=0.$$
\end{proposition}

\begin{proof}
First, consider the case $m=2$. Set
$$f(t_1,...,t_n)=t_0^4+b(t_1^3+\cdots+t_n^3)t_0+(t_1^{4}+\cdots+t_n^{4})+
\sum_{1\leq i_1<i_2<i_3<i_4\leq n} c
t_{i_1}t_{i_2}t_{i_3}t_{i_4}$$ with some constants $b,c\in \mathbb
C^*$. Using Notation \ref{n.f_k}, we have
$$f_1=f_2=0, \; f_{3}=b(t_1^3+\cdots+t_n^3) \mbox{ and }$$
$$f_{4}=(t_1^{4}+\cdots+t_n^{4})+\sum_{1\leq i_1<i_2<i_3<i_4\leq
n}c t_{i_1}t_{i_2}t_{i_3}t_{i_4}.$$  Since the Fermat hypersurface
in $\BP^n$ defined by
 $t_0^{4}+t_1^{4}+\cdots+t_n^{4}=0$ is smooth, the hypersurface
 $Y$
 defined by $f=0$ is smooth  if we choose  general $b$ and $c$.
 Set $x:=[1:0:\cdots:0]$. By Propositions \ref{p.dphi} and
\ref{p.ImOb}, we have
\begin{align*}{\rm Im}(d\mu_x)&= \left\{\sum_{i=1}^{n}v_i\frac{\partial
f_{4}}{\partial t_i}(z) \ \Bigr| \ v_i\in \mathbb C\right\}=
\left\{\sum_{i=1}^n v_i(4z_i^3+\sum_{1\leq i_1<i_2<i_3\leq
n,\forall i_k\neq i}c z_{i_1}z_{i_2}z_{i_3}) \ \Bigr| \ v_i\in
\mathbb C\right\}\end{align*} and
\begin{align*}T_{\mu(x)}(GL(n,\mathbb
C).\mu(x))&= \left\{\sum_{i,j=1}^ns^i_jz_i\frac{\partial
f_{3}}{\partial t_j}(z) \ \Bigr| \ s^i_j\in \mathbb
C\right\}=\left\{\sum_{i,j=1}^ns^i_jz_iz_j^2\ \Bigr|  \ s^i_j\in
\mathbb C\right\} .\end{align*} From this it follows that ${\rm
rank}(d\mu_x)=n$ and $\dim_{\mathbb C} T_{\mu(x)}(GL(n,\mathbb
C).\mu(x))=n^2$. Also if
 there exist $s^i_j$ and $v_i$ such that
$$\sum_{i,j=1}^ns^i_jz_iz_j^2=\sum_{i=n}^nv_i(4z_i^3+\sum_{1\leq i_1<i_2<i_3\leq n,
\forall i_k\neq i} c z_{i_1}z_{i_2}z_{i_3}),$$then  $s^i_j=0$ and
$v_i=0$ for all $i$ and $j$. Therefore
$${\rm Im}(d\mu_{x})\cap T_{\mu(x)}(GL(n,\mathbb C).\mu(x))=0.$$

Next, assume that $m\geq 3$. Pick $$f(t_0, \ldots, t_n) = t_0^{2m}
+b(t_1^{m+1}+\cdots+t_n^{m+1})t_0^{m-1} +
ct_1t_2t_3(t_4^{m-1}+\cdots+t_n^{m-1})t_0^{m-2}+t_1^{2m}+\cdots+t_n^{2m}$$
 with some constants $b,c\in
\mathbb C^*$. Using Notation \ref{n.f_k}, we have
$$f_1=\cdots=f_m=f_{m+3}=\cdots=f_{2m-1}=0, \;
f_{m+1}=b(t_1^{m+1}+\cdots+t_n^{m+1}),$$
$$f_{m+2}=ct_1t_2t_3(t_4^{m-1}+\cdots+t_n^{m-1}) \mbox{ and }
f_{2m}=t_1^{2m}+\cdots+t_n^{2m}.$$  From the  smoothness of the
Fermat hypersurface in $\BP^n$ defined by
 $t_0^{2m}+t_1^{2m}+\cdots+t_n^{2m}=0$, we can see that the hypersurface
 $Y$
 defined by $f=0$ is smooth  for  general $b$ and $c$. Set $x:=[1:0:\cdots:0]$.
 Propositions \ref{p.dphi} and
\ref{p.ImOb} show that
\begin{align*}{\rm Im}(d\mu_x)&=\left\{\sum_{i=1}^{n}v_i\frac{\partial
f_{m+2}}{\partial t_i}(z)\ \Bigr| \ v_i\in \mathbb C\right\}\\&
=\left\{(v_1z_2z_3+v_2z_1z_3+v_3z_1z_2)(z_4^{m-1}+\cdots+z_n^{m-1})+\sum_{i=4}^nv_iz_1z_2z_3z_i^{m-2}\
\Bigr|\ v_i\in \mathbb C\right\}\end{align*} and
\begin{align*}T_{\mu(x)}(GL(n,\mathbb
C).\mu(x))&=\left\{\sum_{i,j=1}^ns^i_jz_i\frac{\partial
f_{m+1}}{\partial t_j}(z)\ \Bigr| \ s^i_j\in \mathbb C \right\}
=\left\{\sum_{i,j=1}^ns^i_jz_iz_j^m\ \Bigr| \ s^i_j\in \mathbb C
\right\} .\end{align*} The condition $m \geq 3$ implies  that
${\rm rank}(d\mu_x)=n$. It is easy to see that $$ \dim_{\mathbb C}
T_{\mu(x)}(GL(n,\mathbb C).\mu(x))=n^2 \mbox{ and }  {\rm
Im}(d\mu_{x})\cap T_{\mu(x)}(GL(n,\mathbb C).\mu(x))=0.$$
\end{proof}

\begin{proof}[Proof of Theorem \ref{t.LR}]
By Proposition \ref{p.isom}, we may prove the corresponding
statement for the morphism $\eta: W \to {\rm Hilb}(\BP^{n-1})$
defined on a neighborhood $W$ of a general point in $\BP^n$ by
$\eta(y):= [\sE^Y_y]$ for $y \in W$.   Since $\sE^Y_x$ is a
complete intersection of multi-degree $(m+1, \ldots, 2m)$ for
general $Y$ and general $x \in \BP^n \setminus Y$, the equation
$B_{m+1}$ of degree $m+1$ is uniquely determined up to
$GL(n,\mathbb C)$-action by the projective equivalence class of
$\mathcal E_x^Y$. Thus it suffices to show that ${\rm
rank}(d\mu_x)=n$ and $d\mu_x(T_x(\BP^n)) \cap
T_{\mu(x)}(GL(n,\mathbb C).\mu(x))=0$ for a general $Y$ and
general $x$. This follows from Proposition \ref{p.trans}.
\end{proof}

\section{Projective connections and rigidity of maps}

\begin{definition}\label{d.connection}
Given a complex manifold $M$ of dimension $n$, the  projectivized
tangent bundle $\pi: \BP T(M) \to M$ is equipped with the
tautological line bundle $\xi \subset \pi^*T(M)$ whose fiber at
$\alpha \in \BP T(M)$ is given by $\hat{\alpha} \subset
T_{\pi(\alpha)}(M)$,  the 1-dimensional subspace corresponding to
$\alpha \in \BP T_{\pi(\alpha)} (M).$ We have the vector subbundle
$\sT \subset T(\BP T(M))$ of rank $n$ whose fiber at $\alpha \in
\BP T(M)$ is given by $$\sT_{\alpha} :=
d\pi_{\alpha}^{-1}(\hat{\alpha})$$ where $d \pi_{\alpha} :
T_{\alpha}(\BP T(M)) \to T_{\pi(\alpha)}(M)$ is the differential
of the projection $\pi.$ A {\em projective connection} on $M$ is a
homomorphism $p: \xi \to \sT$ of vector bundles which splits the
exact sequence of vector bundles on $\BP T(M)$
$$ 0 \longrightarrow T^{\pi} \longrightarrow \sT \longrightarrow
\sT/T^{\pi} \cong \xi \longrightarrow 0$$ where $T^{\pi} \subset
T(\BP T(M))$ is the relative tangent bundle of $\pi$.  Given a
projective connection $p: \xi \to \sT$, the image $p(\xi) \subset
\sT \subset T(\BP T(M))$ is a line subbundle in the tangent bundle
of $\BP(T(M))$ and defines a foliation of rank 1 on $\BP T(M)$.
\end{definition}

\begin{example}\label{e.P} On $\BP^n$, we have a canonical
projective connection $p: \xi \to \sT$ such that the leaves of the
foliation $p(\xi)$ are exactly the tangent directions of lines on
$\BP^n$. We call this the {\em flat projective connection} and
denote it by $p^{\rm flat}$. Let $U$ be a connected complex
manifold of dimension $n$ and let $\varphi: U \to \BP^n$ be an
immersion. Via the biholomorphic morphism $\BP T(U) \cong \BP
T(\varphi(U))$, we have an induced projective connection
$\varphi^* p^{\rm flat}$ on $U$. By the affirmative answer to
Problem \ref{q.Liouville} when $X = \BP^n$ (see the remark after
Problem \ref{q.Liouville}), two immersions $\varphi_i: U \to
\BP^n, i = 1,2,$ are related by a projective transformation, i.e.,
there exists an automorphism $\psi: \BP^n \to \BP^n$ such that
$\varphi_2 = \psi \circ \varphi_1$, if and only if the two
projective connections $\varphi_1^* p^{\rm flat}$ and $
\varphi_2^* p^{\rm flat}$ coincide.
\end{example}

\begin{proposition}\label{p.uniqueconnection}
In the setting of Definition \ref{d.connection}, let $\sC \subset
\BP T(M)$ be a closed subvariety dominant over $M$ such that for a
general point $x \in M$, the fiber $\sC_x \subset \BP T_x(M)$ is
not contained in a quadric hypersurface. Suppose that $p_1, p_2:
\xi \to \sT$ are two projective connections on $M$ such that
$p_1|_{\sC} = p_2|_{\sC}.$ Then $p_1=p_2$.
\end{proposition}

\begin{proof} Since $p_1$ and $p_2$ split the exact sequence
in Definition \ref{d.connection},  the  difference $p_1-p_2$
determines an element $\sigma \in H^0( \BP T(M), T^{\pi} \otimes
\xi^{-1}).$ For a general $x \in M,$ $\sigma_x $ is a section of
$T(\BP T_x(M))\otimes \xi^{-1}$ on the projective space $\BP T_x(M)$.
 The condition $p_1|_{\sC} = p_2|_{\sC}$ implies that $\sigma_x$ vanishes
on the subvariety $\sC_x$. In term of a homogeneous coordinate
system on projective space $\BP^{n-1}$, a nonzero  section of
$T(\BP^{n-1})\otimes \sO_{\BP^{n-1}}(1)$ is represented by a
homogeneous polynomial vector field with quadratic coefficients.
In particular, the zero set of such a section must be contained in
some quadric hypersurface.  By the assumption that $\sC_x$ is not
contained in a quadric hypersurface, we see that $\sigma_x =0$.
Since it is true for a general $x \in M$, we obtain $p_1=p_2$.
\end{proof}

We have the following general version of Theorem \ref{t.germ}. In
fact, Theorem \ref{t.germ} is a corollary of Theorem
\ref{t.general} by  Theorem \ref{t.VMRT}.

\begin{theorem}\label{t.general} Let $X$ be a Fano manifold.
For a general point $x \in X$, we denote by  $\sK_x$ the space of
minimal rational curves through $x$ and  by $\sC_x \subset \BP
T_x(X)$ the VMRT at $x$. Assume that $\sC_x$ is not contained in a
quadric hypersurface in $\BP T_x(X)$.  Let $U \subset X$ be a
connected neighborhood of a general point $x \in X$ and
$\varphi_1, \varphi_2: U \to \BP^n$ be two biholomorphic
immersions such that for any $y \in U$ and any member $C$ of
$\sK_y$, both $\varphi_1(C \cap U)$ and $\varphi_2(C \cap U)$ are
contained in  lines in $\BP^n$. Then there exists a projective
transformation $\psi: \BP^n \to \BP^n$ such that $\varphi_2= \psi
\circ \varphi_1$.
\end{theorem}

\begin{proof}
Let $\sC \subset \BP T(X)$ be the closure of the union of $\sC_x
\subset \BP T_x(X)$ as $x$ varies over general points of $X$. For
a member $C$ of $\sK_x$ and its smooth locus $C^o \subset C$, the
curve $\BP T(C^o) \subset \BP T(X)$ lies in $\sC$. In fact, by the
definition of $\sC$ such curves cover a dense open subset in
$\sC$.

Consider the projective connections $\varphi_i^* p^{\rm flat}$ on
$U$. Let $C$ be a general minimal rational curve intersecting $U$.
Since $\varphi_1(C \cap U)$ and $\varphi_2(C \cap U)$ are
contained in lines in $\BP^n$, the difference $$\varphi_1^* p^{\rm
flat} -\varphi_2^* p^{\rm flat} \in H^0(\BP T(U), T^{\pi} \otimes
\xi^{-1}),$$ in the notation of Definition \ref{d.connection} with
$M=U$, vanishes along the Riemann surface $\BP T(C^o) \cap \BP
T(U)$. Since such Riemann surfaces cover a dense open subset in
$\sC \cap \BP T(U)$, the two projective connections must agree on
$\sC \cap \BP T(U)$. Applying Proposition \ref{p.uniqueconnection}
with $M=U$, we conclude $\varphi_1^* p^{\rm flat} = \varphi_2^*
p^{\rm flat}.$ As mentioned in Example \ref{e.P}, this implies the
existence of a projective transformation $\psi$ satisfying
$\varphi_2 = \psi \circ \varphi_1.$ \end{proof}

\begin{proof}[Proof of Theorem \ref{t.map}]
Putting $m=n-1$ in the proof of Proposition \ref{p.isom}, we see
that minimal rational curves on $X_i, i=1,2,$ have trivial normal
bundles and rational curves through general points with trivial
normal bundles are minimal rational curves. By Proposition 6 of
\cite{HM03} (also cf. Theorem 3.1 (iv) in \cite{S}), for a general
minimal rational curve $C \subset X_2$, each irreducible component
of $f^{-1}(C)$ is a minimal rational curve in $X_1$. In other
words, $f$ sends minimal rational curves of $X_1$ through a
general point to those of $X_2$.  Putting $$\hat{X}= X_1, X = X_2,
g=f, \phi= \phi_2, \mbox{ and } h = \phi_1$$ in Corollary
\ref{c.finite}, we see that $\phi_1 = \psi \circ \phi_2 \circ f$
for some projective transformation $\psi$. Thus $f$ must be
birational, and hence an isomorphism.
\end{proof}

\begin{proof}[Proof of Theorem \ref{t.Liouville}]
Applying Theorem \ref{t.germ} to $\varphi:= \phi_2 \circ \gamma:
U_1 \to \phi_2(U_1) \subset \BP^n$, we have a projective
transformation $\psi \in {\rm Aut}(\BP^n)$ such that $\psi \circ
\phi_1|_{U_1} =  \phi_2 \circ \gamma$. By the assumption on
$\gamma$ and Proposition \ref{p.mrc}, we have $d\psi(\sE^{Y_1}_x)
= \sE^{Y_2}_{\psi(x)}$ for $x \in \phi_1(U_1)$. By Proposition
\ref{p.unique}, this implies $\psi(Y_1)= Y_2.$ Thus replacing
$Y_1$ by $\psi(Y_1)$ and $\phi_1$ by $\psi \circ \phi_1$, we may
assume that $Y_1=Y_2$ and $\phi_1(U_1) = \phi_2(U_2)$. From Lemma
\ref{l.unique},  there exists a biregular
 morphism $\Gamma: X_1 \to X_2$ with $\Gamma|_{U_1} = \gamma$. \end{proof}

\end{document}